\documentclass{cimart}

\usepackage{amssymb,amscd,amsthm,latexsym, amsmath}
\usepackage[all]{xy}

\newcommand{\Cal}[1]{{\mathcal #1}}

\newcommand{\Gp}{\mathsf{Gp}}
\newcommand{\Set}{\mathsf{Set}}

\DeclareMathOperator{\Aut}{Aut}

\renewcommand{\phi}{\varphi}
\DeclareMathOperator{\op}{\rm op}

\DeclareMathOperator{\End}{End}

\title{Skew braces, near-rings, skew rings, dirings
    }

\author{Alberto Facchini
    }

\authorinfo[A. Facchini]{%Department of Mathematics ``Tullio Levi-Civita'', 
University of Padova, Italy}{%
   facchini@math.unipd.it }

\abstract{%
    We introduce a new point of view to present classical notions related to set-theoretic solutions of the Yang-Baxter equation: left skew braces, dirings, and left skew rings. The idea is to replace the single multiplication on a left near-ring by two operations, one associative and the other left distributive. Two algebraic structures naturally appear: left skew rings and left weak rings, whose categories turn out to be canonically isomorphic.
    }



\VOLUME{33}
\YEAR{2025}
\ISSUE{3}
\NUMBER{14}
\DOI{https://doi.org/10.46298/cm.16621}

\begin{document}

% Here is where the main text should be typed:
	\section{Introduction}
	
	In this article, we develop, from a new point of view, classical notions connected to set-theoretic solutions of the Yang-Baxter equation. Our aim  is to highlight some properties concerning left skew braces, left near-rings, and left skew rings. The basic idea is to replace an operation on an additive group $(G,+)$ with the difference of two operations on $(G,+)$, one associative and the other left distributive. This is the same technique by which, given a module 
$M$ over a commutative ring 
$k$ in which $2$ is invertible, any $k$-bilinear operation on
$M$ can be expressed as the sum of two $k$-bilinear operations, one commutative and the other anticommutative.

Here, we start from the elementary notions related to left near-rings (Section~\ref{basic}). A left near-ring is equipped with a multiplication operation that is associative and left distributive. (In this paper, whenever we say that a binary operation $\cdot$ on a group $(G,+)$ is left distributive, we mean that the operation $\cdot$ is left distributive over $+$, that is, $a\cdot(b+c)=a\cdot b+a\cdot c$ for every $a,b,c\in G$.) Examples of near-rings include the set
$M(G)$ of all mappings $G\to G$, where $G$ is an additive group (these form a {\em right} near-ring), and the set
$B(G)$ of all binary operations $G\times G\to G$,
again with $G$ an additive group (\hspace{-.01cm}\cite[Theorem~1.2]{LPOMR} and \cite[Theorem~2.4.5]{OM}). In particular, $B(G)$ is a left near-ring with a two-sided identity. The identity is the operation $\pi_1$ on $G$, where $\pi_1\colon G\times G\to G$ is the first projection, that is, the operation defined by $a\pi_1 b=a$ for every $a,b\in G$ (Section~\ref{4}).

We consider all possible ways of expressing the operation $\pi_1$
as the difference of two operations $\circ$ and $\cdot$, where $\circ$ 
 is a binary associative operation and $\cdot$ is a left distributive operation. This naturally gives rise to {\em two} structures closely resembling that of a left near-ring: {\em left skew rings}, defined in the same way as left near-rings except for the fact that left distributivity is replaced by left skew distributivity; and {\em left weak rings}, also defined as left near-rings except that associativity is replaced by left weak associativity. The two categories of left skew rings and left weak rings turn out to be isomorphic (Theorem~\ref{main}). 

When braces appeared in the mathematical literature, they were considered a generalization of  (Jacobson) radical rings (\hspace{-0.01cm}\cite{AV} and \cite{V}). The radical of any local near-ring is a skew-brace \cite[Section~5]{RumpJAA2019}, and the skew-rings introduced in \cite[Corollary of Proposition~1]{RumpJAA2019}
are a common generalization of skew braces and (unital) near-rings in a very natural way \cite[Section~5]{RumpJAA2019}.
 The difference between our left skew rings $(R,+,\circ)$ and those defined by Rump is that we do not require the existence of a two-sided identity for the operation $\circ$.

After completing this article, we became aware that a similar idea can also be found in a recent paper by Bai, Guo, Sheng and Tang, concerning post-groups \cite{Ann}. However, the point of view in \cite{Ann} is quite different from ours, particularly in regard to the idea of expressing the identity operation 
$\pi_1$ as the difference of an associative operation and a left distributive one. Nevertheless, we chose to retain here the term ``left weak associativity'' used by Bai, Guo, Sheng and Tang.
	
	In this note, when we say ``algebra'' we mean ``algebra'' in the sense of Universal Algebra, and we will sometimes denote algebras in an informal way. For instance, an additive group will be denoted both in the form $(G,+,-,0_G)$ as one usually correctly does in Universal Algebra and in the more common form $(G,+)$. Similarly, a (near-)ring will be denoted in both forms $(N,+,-,0_N,\cdot)$ and $(N,+,\cdot)$. Also, when we say that a binary operation $\cdot$ on an additive group $(G,+)$ is ``left distributive'', we mean that $\cdot$ distributes over addition of the group $G$, i.e., that $a(b+c)=ab+ac$ for every $a,b,c\in G$.

		\section{Basic terminology on near-rings, skew braces and digroups}\label{basic}
	
	\subsection{Left near-rings} \label{4.0}  Let $N$ be a left near-ring, that is, an algebra $(N, +,-,0,\cdot)$ for which
\begin{enumerate}
    \item[(a)]  $(N, +,-,0)$ is a (not-necessarily abelian) group;
    \item[(b)] the binary operation $\cdot$ is {\em associative}; and
    \item[(c)] {\em left distributivity} holds, that is, $a(b+c)=ab+ac$ for every $a,b,c\in N$.
\end{enumerate}

Here and in the following, whenever we have a binary operation $\cdot$, we will denote the product of $a$ and $b$ by either $a\cdot b$ or $ab$.

\smallskip

 It is easily seen that an algebra $(N, +,-,0,\cdot)$ is a left near-ring if and only if $(N, +,-,0)$ is a group, $(N, \cdot)$ is a semigroup and, for every $a\in N$, the mapping $\mu^N_a\colon N\to N$, defined by
$\mu^N_a(b)=a\cdot b$ for every $b\in N$, is a group endomorphism of the group $(N,+)$. Equivalently, $(N, +,-,0)$ is a group, $(N, \cdot)$ is a semigroup, and
the mapping $\mu^N\colon (N,\cdot)\to \End_\Gp(N,+)$, defined by $\mu^N\colon a\mapsto \mu^N_a$, is a semigroup morphism. 

For a left near-ring, $a0 = 0$  and $a(-b)=-(ab)$ always hold (because each $\mu^N_a$ is a group endomorphism of the group $(N,+)$), while $0a = 0$  and $(-a)b=-(ab)$ do not necessarily hold. More precisely, the endomorphism $\mu^N_0\colon a\in N\mapsto 0a$ is an idempotent group endomorphism of the group $(N,+)$, so that $(N,+)$ decomposes as the semidirect sum of the kernel $\ker(\mu_0^N)$ of $\mu_0^N$ and its image $\mu_0^N(N)$. That is,
every near-ring $N$ can be decomposed as a semidirect sum $N = N_0 \rtimes N_c$ as an additive group, where $$N_0 := \ker(\mu_0^N)=\{\,a\in N \mid 0a = 0\,\}$$ is
the {\em $0$-symmetric part} of $N$ and $$N_c := \mu_0^N(N)=\{\,0a\mid a\in N\, \}=\{\,a\in N\mid 0a=a\, \}=\{\,a\in N \mid ba = a\ \mbox{\rm for all }b\in N\,\}$$ is the {\em constant part} of $N$. Here, with the notation $N = N_0 \rtimes N_c$, we mean that $N_c$ acts on $N_0$ via the group morphism $\varphi\colon (N_c,+)\to\Aut_\Gp(N_0,+)$ defined by $\varphi_a(b)=a+b-a$ for every $a\in N_c$, $b\in N_0$. Notice that
$(N,\cdot,0)$ is a semigroup with right zero $0$, and $0a$ is also a right zero for $(N,\cdot)$ for every $a\in N$. That is, $N_c$ is the set of all right zeros for $(N,\cdot)$.

Left near-rings form a semi-abelian variety that is pointed and ideal-determined, i.e., there is a natural one-to-one correspondence between congruences and ideals. Here:

\begin{definition}\label{bip} An {\em ideal} $A$ of a left near-ring $(N, +,-,0,\cdot)$ is a normal subgroup of the group $(N,+,-,0)$ such that $ma\in A$ and $(a + m)n-mn\in A$ for every $a\in A $ and every $m,n\in N$. (Notice that if $A$ is a normal subgroup, then $m+A=A+m$ for every $m\in N$, and therefore the condition ``$(a + m)n-mn\in A$ for every $a\in A $ and every $m,n\in N$'' is equivalent to  ``$(m+a)n-mn\in A$ for every $a\in A $ and every $m,n\in N$''.)\end{definition}

The subgroups $N_0$ and $N_c$ are subnear-rings of $N$, but the normal subgroup $N_0$ of the additive group $N$ is not necessarily an ideal of the left near-ring $N$ (see Example~\ref{mmmm}). A left near-ring $N$ is {\em $0$-symmetric} if $0a = 0$ for every $a\in N$. 
Let us try to be very careful with the semidirect-sum of left near-rings and the semidirect-sum decomposition of the additive group~$N$ as $N = N_0 \rtimes N_c$. As we have already said, $N_0$ is a normal subgroup of the additive group $N$, not necessarily an ideal of~$N$, and is a subnear-ring of $N$; and $N_c$ is a subnear-ring of $N$, not necessarily a normal subgroup. This occurs because the idempotent endomorphism $\mu_0^N\colon N\to N$, whose kernel  is $N_0$, and whose image is $N_c$,  is an additive group endomorphism, but not a left near-ring endomorphism in general (it need not respect multiplication).  

\begin{example}\label{mmmm} Let us give an example of a left near-ring $N$ that shows that $N_0$ is not necessarily an ideal of $N$ and $N_c$ is not necessarily a
normal subgroup of $(N,+)$. Let $G$ be any non-abelian group, written additively, and consider the set $M(G)$ of all mappings $G\to G$. If we write mappings on the right, then $M(G)$ becomes a left near-ring whose addition is point-wise addition and whose multiplication is composition. This is the most standard example of a left near-ring, because every left near-ring is a sub-near-ring of $M(G)$ for some group $G$. It is easily seen that the idempotent endomorphism $\mu_0\colon M(G)\to M(G)$ associates to every mapping $f\in M(G)$ the mapping $\mu_0(f)\colon G\to G$ constantly equal to $(0)f$. The kernel $M_0(G)$ of $\mu_0$ is the $0$-symmetric part of $M(G)$. It consists of all mappings $f\in M(G)$ such that $(0)f=0$. The image of $\mu_0$, the constant part of $M(G)$, is the sub-near-ring $M_c(G)$ consisting of all constant mappings $G\to G$. Then $M_0(G)$ is not an ideal of $M(G)$. For instance, let $a$ be the identity mapping of $G$, so that $a\in M_0(G)$, and $m\in M(G)$ any mapping $G\to G$ constantly equal to a non-zero element $g\in G$. Then $(0)ma=(0)m=g$, so $ma\notin M_0(G)$. This shows that $M_0(G)$ is not an ideal of $M(G)$. Let us prove that $M_c(G)$ is not a normal subgroup of the additive group~$M(G)$. Since $G$ is not abelian, there exist two elements $g,h\in G$ with $g+h\ne h+g$. Let  $m\in M(G)$ the mapping $G\to G$ constantly equal to $g$, and $a$ be the identity mapping of~$G$. Then $m\in M_c(G)$, but $a+m-a\notin M_c(G)$, because $a+m-a$ is not constant, for instance 
\[
(0)(a+m-a)=(0)a+(0)m-(0)a=g \ \text{and} \  (h)(a+m-a)=(h)a+(h)m-(h)a=h+g-h\ne g.
\]
Hence $M_c(G)$ is not a normal subgroup of $M(G)$.\end{example}

We need some further basic notions concerning left near-rings, but since these notions apply not only to near-rings, but more generally to any $\Omega$-group in the sense of Higgins, in the next subsection we present  these facts in the setting of $\Omega$-groups.

\subsection{$\Omega$-groups}

All the algebras we consider in this paper (left near-rings, skew braces, digroups, left skew rings, left weak rings, left dirings) are {\em  $\Omega$-groups} $(G,+,-,0, p_a\mid a\in\Omega)$ in the sense of Higgins  \cite{26}, 
that is, they are varieties of algebras which 
have amongst their operations and identities those of the variety of groups and are pointed (i.e., they have exactly one constant, the zero element of the group, which forms a one-element subalgebra; this simply means that $p_a(0,0,0,\dots,0)=0$ for every $a\in\Omega$).

Ideals  of an  $\Omega$-group $(G,+,-,0, p_a\mid a\in\Omega)$ are defined in \cite[p.~373]{26} in a rather technical way, making use of ``ideal words'', but one easily sees \cite[Theorem~4A]{26} that ideals of $G$ are the normal subgroups $H$ of the additive group $(G,+)$ such that $$p_a(g_1,g_2,\dots, g_{i-1}, h+g_i, g_{i+1}\dots, g_n)\equiv p_a(g_1,g_2,\dots, g_n) \pmod{H}$$ for all operations $p_a$ of $G$ ($a\in \Omega$), where $n$ is the arity of $p_a$, $i=1,2,\dots,n$, $h\in H$ and $g_1,g_2,\dots, g_n\in G$. Equivalently, ideals of $G$ are the normal subgroups $H$ of  $(G,+)$ for which the corresponding congruence $\equiv_H$ on the group $G$ is compatible with all the operations $p_a$, $a\in\Omega$; that is, the normal subgroups $H$ of  $(G,+)$ for which the corresponding congruence $\equiv_H$ on the group $G$ is a congruence of the algebra $G$.
That is, in other words, ideals of $G$ are the equivalence classes $[0]_\omega$ of the zero of the additive group $G$ modulo some congruence $\omega$ of the algebra $G$. One sees immediately that:

\begin{lemma}\label{vil} Let $(G,+,-,0, p_a\mid a\in\Omega)$ be an $\Omega$-group. There is a one-to-one correspondence between the set of all congruences of the algebra $G$ and the set of all ideals of $G$.\end{lemma}
		\begin{proof} 
It is well known that, for the group $(G,+,-,0)$, there is a one-to-one correspondence $\Phi$ of the set $\Cal C_{(G,+)}$ of all congruences of the group $G$ onto the set of all normal subgroups of $G$. It associates with any congruence $\sim$ of $(G,+)$ the congruence class $[0]_\sim$ of $0$ modulo $\sim$. If we restrict this correspondence $\Phi$ to the set $\Cal C_{(G,+,p_a)}$ of all congruences of the algebra $(G,+,-,0, p_a\mid a\in\Omega)$, we get a bijection of $\Cal C_{(G,+,p_a)}$ onto its image $\Phi(\Cal C_{(G,+,p_a)})$. It maps any congruence $\omega\in \Cal C_{(G,+,p_a)}$ to $[0]_\omega$. But, by definition, the image 
\[
\Phi(\Cal C_{(G,+,p_a)})=\{\, [0]_\omega\mid \omega\in \Cal C_{(G,+,p_a)}\,\}
\]
is the set of all ideals of the algebra $(G,+,-,0, p_a\mid a\in\Omega)$.
\end{proof}

The one-to-one correspondence in Lemma~\ref{vil} is clearly order-preserving and order-reflecting, so that it is a lattice isomorphism. It is easily seen that the sum of two ideals and the intersection of any family of ideals of an $\Omega$-group $G$ are ideals. (For the sum: If $H$ and $K$ are two ideals of $G$, then $H+K$ is a normal subgroup of $(G,+)$. In order to show that $H+K$  is an ideal of the $\Omega$-group $(G,+,-,0, p_a\mid a\in\Omega)$, we must show that the group congruence $\equiv_{(H+K)}$ is compatible with all the operations $p_a$. To this end, it suffices to show that if $g_1,\dots,g_n\in G$ and $g_i\equiv g'_i\pmod{H+K}$ for some $i=1,\dots,n$ and some $g'_i\in G$, then $p_a(g_1,\dots,g_n)\equiv p_a(g_1,\dots,g'_i,\dots, g_n)\pmod{H+K}$. From $g_i\equiv g'_i\pmod{H+K}$, we get that $g'_i=g_i+h+k$ for suitable elements $h\in H$ and $k\in K$. Then $$p_a(g_1,\dots,g_i,\dots,g_n)\equiv p_a(g_1,\dots,g_i+h,\dots, g_n)\pmod{H}$$ and $$p_a(g_1,\dots,g_i+h,\dots,g_n)\equiv p_a(g_1,\dots,g_i+h+k,\dots, g_n)\pmod{K},$$ that is, 
\[
p_a(g_1,\dots,g_i,\dots,g_n)- p_a(g_1,\dots,g_i+h,\dots, g_n)\in{H}
\]
and 
\[
p_a(g_1,\dots,g_i+h,\dots,g_n)- p_a(g_1,\dots,g_i+h+k,\dots, g_n)\in{K}.
\]
Summing up, we get that $p_a(g_1,\dots,g_i,\dots,g_n)-  p_a(g_1,\dots,g_i+h+k,\dots, g_n)\in{H+K}$.)
Therefore in the complete lattice of all the ideals of an $\Omega$-group we have that $H\vee K=H+K$ and $H\wedge K=H\cap K$.

Let us assume that we have an idempotent endomorphism $e$ of an $\Omega$-group given by $(G,+,-,0, p_a\mid a\in\Omega)$, where endomorphism means an algebra endomorphism, that is, a mapping $e\colon G\to G$ that respects both the addition $+$ and all the operations $p_a\ (a\in\Omega)$. Then the kernel $K$ of $e$, that is, the inverse image $e^{-1}(0)$ of $0$, is an ideal of $G$, and the image $H:=e(G)$ of $e$ is a subalgebra of $G$, that is, a subgroup of $(G, 0,+,-)$ closed for all the operations $p_a$. Since the algebra endomorphism $e$ is, in particular, an idempotent group endomorphism of the group $(G, +,-,0)$, we have that $G$ has an inner semidirect-product decomposition $G=K\rtimes H$ as an additive group. A kind of converse also holds, as the next proposition shows.
	
	\begin{proposition}\label{xxx} Let $(G,+,-,0, p_a\mid a\in\Omega)$ be an $\Omega$-group, $e$ be an idempotent group endomorphism of $(G,+)$, and $K$ and $H$ be the kernel and the image of $e$ respectively, so that $G=K\rtimes H$ as an additive group. The following two conditions are equivalent:

    \begin{enumerate}
        \item[(a)] The kernel $K$ is an ideal of the algebra $G$ and $H$ is a subalgebra of $G$.
        \item[(b)] $e$ is an algebra endomorphism of $(G,+,-,0, p_a\mid a\in\Omega)$.
    \end{enumerate}
\end{proposition}
	
More generally, we have the following result of Universal Algebra:

	\begin{proposition}\label{xxxx} Let  $A$ be an algebra and $e\colon A\to A$ be an idempotent mapping. Let $\sim_e$ be the equivalence relation on the set $A$ defined, for every $a,a'\in A$, by $a\sim_ea'$ if $e(a)=e(a')$, and let $B=e(A)$ be the image of $e$. The following two conditions are equivalent:

    \begin{enumerate}
        \item[(a)] $\sim_e$ is a congruence of the algebra $A$, and $B$ is a subalgebra of $A$.

        \item[(b)] $e$ is an algebra endomorphism of $A$.
    \end{enumerate}
\end{proposition}
	
\begin{proof} 
(a)${}\Rightarrow{}$(b) Suppose that (a) holds. In order to prove (b), we must prove that, for any operation $f\in F$ of the algebra $(A,F)$, we have $f(e(a_1),\dots,e(a_n))=e(f(a_1,\dots,a_n))$. Here $n$ is the arity of $f$ and $a_1,\dots,a_n\in A$. Now $e=e^2$, so that, for each $i=1,2,\dots,n$, we know that $e(a_i)=e(e(a_i))$. Thus $a_i\sim_ee(a_i)$. By the compatibility between $f$ and $\sim_e$, we get that $f(a_1,\dots,a_n)\sim_ef(e(a_1),\dots,e(a_n))$. But $B=e(A)$ is a subalgebra of $A$, thus $f(e(a_1),\dots,e(a_n))\in B$, so that $e(f(e(a_1),\dots,e(a_n)))=f(e(a_1),\dots,e(a_n))$. Therefore $f(e(a_1),\dots,e(a_n))=e(f(e(a_1),\dots,e(a_n)))=e(f(a_1,\dots,a_n))$, as desired.

(b)${}\Rightarrow{}$(a) is trivial.\end{proof}

\begin{corollary}\label{2.1} The following conditions are equivalent for an $\Omega$-group $G$, an ideal $K$ of $G$ and a subalgebra $H$ of $G$:
\begin{enumerate}
    \item[(a)] $G=K+H$ and $K\cap H=0$. 
    \item[(b)] For every $g\in G$, there is  a unique pair $(k,\, h)\in K\times H$ such that $g=k+h$.
    \item[(c)] For every $g\in G$, there is  a unique pair $(k',\, h)\in K\times H$ such that $g=h+k'$.
    \item[(d)] There is an idempotent algebra endomorphism of $G$ with kernel $K$ and image $H$.
    \item[(e)]   There is an algebra morphism of $G$ onto $H$ that is the identity on $H$ and has kernel~$K$.
\end{enumerate}\end{corollary}

\begin{proof} The equivalence of (a), (b) and (c) holds not only for $\Omega$-groups, but more generally for all groups $G$, normal subgroups $K$ of $G$, and subgroups $H$ of $G$.

(a)${}\Rightarrow{}$(d) follows from Proposition~\ref{xxx}.

(d)${}\Rightarrow{}$(e) ${}\Rightarrow{}$(a) are now trivial.
\end{proof}

If 
the equivalent conditions of Corollary~\ref{2.1} are satisfied, we say that the $\Omega$-group $G$ is the {\em inner semidirect product} of its ideal $K$ and its subalgebra $H$. 

\bigskip

Let $(G,+,-,0, p_a\mid a\in\Omega)$ be an $\Omega$-group. It is well known that, for the group $(G,+)$, there is a one-to-one correspondence between the set of all group endomorphisms of $G$ and the set of all pairs $(K,H)$ with $K$ a normal subgroup of $G$, $H$ a subgroup of $G$, $K+H=G$ and $K\cap H=0$. Restricting this correspondence to the set of all algebra endomorphisms of $G$, one sees from 
Proposition~\ref{xxx}, that  there is a one-to-one correspondence between the set of all algebra endomorphisms of $G$ and the set of all pairs $(K,H)$ with $K$ an ideal of $G$, $H$ a subalgebra of $G$, $K+H=G$, and $K\cap H=0$. 

\begin{remark} I am grateful to Professor George Janelidze who has remarked that the results in this subsection are true for any variety of $\Omega$-groups, and not only for left near-rings, as I had proved in a previous version of this paper. 
Let me stress here that, as far as semidirect products are concerned, our terminology here and in our previous articles is different from the terminology introduced by Bourn and Janelidze in \cite{BJ}. In order to construct semidirect products, they define actions of an object of a semi-abelian category on another object \cite{BJ}. A more systematic theory of internal object actions was developed in \cite{BJK}.

We construct {\em  inner semidirect-product decompositions} of any algebra $A$ (in the sense of Universal Algebra), making use of any subalgebra $B$ of $A$ and any congruence $\omega$ on $A$ such that the intersection of $B$ with any congruence class modulo $\omega$ is a singleton \cite{semi}. Inner semidirect-product decompositions $A=B \ltimes\omega$ of an algebra $A$ turn out to be in one-to-one correspondence with the idempotent endomorphisms of the algebra $A$. In the case of $\Omega$-groups, in which congruences correspond to ideals, the condition that the intersection of $B$ with any congruence class modulo $\omega$ is a singleton is equivalent to $B\cap K=0$ and $B+K=A$, where $K$ is the equivalence class of $0$ modulo the congruence $\omega$. (The equivalence class of any element $g\in G$ modulo $\omega$ is the coset $g+K$.)

{\em Outer semidirect products} (or, better, {\em semidirect-product extensions} of $B$) are constructed from and parametrized by functors $F\colon \Cal C_B\to \Set_*$ from a suitable category $\Cal C_B$ containing $B$, called the {\em enveloping category} of $B$ or the {\em term category }of $B$, to the category $\Set_*$ of pointed sets \cite{semi}. The objects of $\Cal C_B$ are the $n$-tuples $(b_1,\dots,b_n)$ of elements of $B$, with $n\ge 1$ and $b_1,\dots,b_n\in B$. A morphism $(b_1,\dots,b_n)\to (c_1,\dots,c_m)$ in $\Cal C_B$ is any $m$-tuple $(p_1,\dots,p_m)$ of $n$-ary terms \cite[Definitions~10.1 and~10.2]{BS} such that $p_j(b_1,\dots,b_n)=c_j$ for every $j=1,2,\dots,m$. The composition of morphisms is defined by 
$$(q_1,\dots,q_r)\circ(p_1,\dots,p_m)=(q_1(p_1,\dots,p_m), q_2(p_1,\dots,p_m),\dots, q_r(p_1,\dots,p_m)).$$ See \cite[Section~4]{semi}.
\end{remark}

\subsection{Left skew braces} \label{4.1} Left skew braces are algebraic structures that have received considerable attention in the past decade because of their relation with set-theoretic solutions of the Yang-Baxter equation.  A {\sl left skew brace} \cite{GV} is
		a triple $(A, +,\circ)$, where $(A, + ) $ and $(A,\circ)$ are groups (not necessarily abelian) such that 
		
		\medskip 
		
		\noindent {\em (left skew distributivity)}\qquad\qquad $a\circ (b + c) = (a\circ b)-a+ ( a\circ c)$
		
		\medskip 
		
		\noindent 
		for every $a,b,c\in A$. Here $-a$ denotes the inverse (opposite) of $a$ in the group $(A,+)$. The inverse of $a$ in the group $(A,\circ)$ will be denoted by $a^{-1}$. 
		It is easy to prove that in a left skew brace the identities $1_{(A,+)}$ and $1_{(A,\circ)}$ of the two groups $(A, +) $ and $(A,\circ)$ coincide. 
			
			Clearly, left skew braces $(A, +,\circ)$ can be considered $\Omega$-groups in two possible ways, using either $(A, + ) $ or $(A,\circ)$ as ``basic group structure'', but it is easy to convince oneself that the most convenient way to view $(A, +,\circ)$ as an $\Omega$-group  is to use its additive structure $(A, + ) $ as basic group structure.
			
			A {\em set-theoretic solution of the
Yang-Baxter equation} (Drinfeld
\cite{Drinfeld}) is a pair $(X,r)$, where
$X$ is a set, $r\colon X\times X\to X\times X$ is a bijection, and
the two mappings $$
(r\times id)(id\times r)(r\times id)\colon X\times X\times X\to X\times X\times X$$ and $$(id\times r)(r\times id)(id\times r)\colon X\times X\times X\to X\times X\times X$$ coincide.
It is convenient to write $r(x,y)$ as $(\sigma_{x}(y),\tau_{y}(x))$
with $\sigma_{x},\tau_{x}\colon X\to X$. A solution $(X,r)$ is {\em non-degenerate}
if the mappings $\sigma_{x}$ and $\tau_{x}$ are bijective for every
$x\in X$.
For every left skew brace $(A,+,\circ)$, the
pair $(A,r_{A})$, where $r_{A}$ is the mapping 
\[
r_{A}\colon A\times A\to A\times A,\quad r_{A}(x,y)=(-x+(x\circ y),(-x+(x\circ y))^{-1}\circ x\circ y),
\]
is a non-degenerate set-theoretic solution of the Yang-Baxter equation. Conversely, if $(X, r)$ is a non-degenerate solution of the Yang-Baxter equation and $G(X, r) $ denotes the {\em structure group}
of $(X, r)$, defined as the group generated by the set $\{\,e_x\mid x\in X\,\}$
and relations $e_xe_y = e_ue_v$ whenever $r(x, y) = (u, v)$, there exists a unique skew left brace structure on $G(X, r)$ such that
$r_{G(X,r)}(\iota\times\iota) = (\iota\times\iota)r$. See \cite[Theorem 9]{38}, \cite[Theorem 2.7]{49} and \cite[Theorems~3.1 and~3.5]{AV}.
			
	\subsection{Digroups}\label{4.2}  A {\sl digroup} \cite{BFP} is
		a triple $(D, +,\circ)$, where $(D, + ) $ and $(D,\circ)$ are groups for which the identities $1_{(D,+)}$ and $1_{(D,\circ)}$ of the two groups $(D, +) $ and $(D,\circ)$ coincide. Hence, left skew braces are digroups. For digroups, it  is possible to define, for every element $a$ of $D$, the mapping $\lambda^D_a\colon D\to D$ setting
$$\lambda^D_a(b)=-a+(a\circ b)$$ for every $b\in D$. Then every $\lambda^D_a$ is a bijection that sends $1$ to $1$, that is, every $\lambda^D_a$ is an automorphism of $(D,1)$ in the category $\Set_*$ of pointed sets. The mapping $$\lambda^D\colon D\to \Aut_{\Set_*}(D,1),\qquad \lambda^D\colon a\mapsto\lambda^D_a,$$ is a morphism 
$\lambda^D\colon (D,1)\to (\Aut_{\Set_*}(D,1),\mathrm{Id}_D)$ in the category $\Set_*$ of all pointed sets, of $(D,1)$ into the automorphism group $\Aut_{\Set_*}(D,1)$ of all automorphisms of  the pointed set $(D,1)$. 
The digroup $(D,+,\circ)$ turns out to be a left skew brace if and only if $\lambda^D$ is a group morphism of $(D,\circ)$ into the group $\Aut_\Gp(D,+)$ or, equivalently, when $\lambda^D_a$ is a group automorphism of the group $(D,+) $ for all $a\in D$.  

The term ``digroup" to indicate these algebraic structures was first used by Bourn in \cite{BournTAC2000}, but
digroups were noticed for the first time in 1997, when George Janelidze, in joint work with Dominique Bourn, observed that the variety of digroups has the Huq commutator different from the Smith commutator  (see the Introduction to the paper \cite{BournTAC2000}, 
or the Introduction to the paper \cite{BFP}, or \cite[p.~356]{BB}, or \cite[p.~36]{BournTAC2004}). In Categorical Algebra,  the notions of Huq commutator, Smith commutator, abelian object and Bourn's strong protomodularity are strictly related, in particular in the setting of semi-abelian categories. 
It is well known that  the (Huq=Smith) condition is satisfied for groups, non-unital rings and Lie algebras, whereas the Huq and the Smith commutators need not coincide in the varieties of digroups (Janelidze 1997), near-rings \cite{JMV2016}, and loops \cite{loops}.

	\section{Left skew rings}\label{4.3}  
	
	Looking at the definitions of near-rings and skew braces above, one sees that a very natural related structure is the following:  A {\sl left skew ring}  is
	an algebra $(R, +,-,0_R,\circ)$, where $(R, +,-,0_R) $ is a group, $\circ$ is a binary operation on $R$, and the following two axioms are satisfied:
    
\begin{equation*}	
(associativity \ of \ the \ operation \ \circ) \qquad \qquad a\circ (b \circ c) = (a\circ b)\circ c,
\end{equation*}
		
and

\begin{equation}\label{equation section 3}
\hspace{0.3cm}(left \ skew \ distributivity) \qquad \qquad a\circ (b + c) = (a\circ b)-a+ ( a\circ c)
\end{equation}
for every $a,b,c\in R$. Our left skew rings are the analog of the skew rings introduced in \cite[Corollary of Proposition~1]{RumpJAA2019}, but without requiring the existence of a two-sided identity for the multiplication $\circ$. 

		\begin{proposition}\label{2.3} If $(R, +,\circ)$ is a left skew ring, then:

        \begin{enumerate}
            \item[(a)] The identity $0_R$ of the group $(R, +) $ is a right identity for the semigroup $(R,\circ)$.
            \item[(b)] For every $a\in R$, the mapping $\lambda^R_a\colon R\to R$, defined by
$\lambda^R_a(b)=-a+(a\circ b)$ for every $b\in R$, is a group endomorphism of the group $(R,+)$.
\item[(c)] The mapping $\lambda^R\colon (R,\circ)\to \End_\Gp(R,+)$, defined by $\lambda^R\colon a\mapsto \lambda^R_a$, is a semigroup morphism. 
        \end{enumerate}
		
		Conversely, if $(R, +)$ is a group, $\circ$ is a binary operation on $R$ relatively to which $(R,\circ)$ is a semigroup, and {\rm (b)} holds, then $(R, +,\circ)$ is a left skew ring.\end{proposition}
		
		\begin{proof} (a) follows from Identity  \eqref{equation section 3} replacing $c$ with the identity $0_R$ of the group $(R,+)$. 
		
		(b) is equivalent to $\lambda^R_a(b+c)=\lambda^R_a(b)+\lambda^R_a(c)$ for every $a,b,c\in R$, and this is trivially equivalent to  Identity \eqref{equation section 3}. 
		
		As far as (c) is concerned, we have that $\lambda_a\circ\lambda_b=\lambda_{a\circ b}$, because for every $c\in R$, 
        \begin{align*}
        (\lambda_a\circ\lambda_b)(c)&=\lambda_a(\lambda_b(c))=\lambda_a(-b+b\circ c)=-\lambda_a(b)+\lambda_a(b\circ c)\\ 
        &=-(-a+a\circ b)-a+a\circ b\circ c=-(a\circ b)+a-a+a\circ b\circ c\\ 
        &=-(a\circ b)+a\circ b\circ c=\lambda_{a\circ b}(c).
        \end{align*}
        
        This completes the proof of (c). The converse is easy, similar to the proof of (b).
		\end{proof}
		
		Clearly,  left skew braces are exactly the left skew rings that are digroups.
		
		As always in Universal Algebra, left skew ring homomorphisms are defined to be the mappings that preserve addition $+$ and multiplication $\circ$. As we have already said, all the algebraic structures considered in Sections \ref{basic} and \ref{4.3}, that is near-rings, skew braces, digroups, and skew rings, form varieties in
the sense of Universal Algebra. They are varieties of $\Omega$-groups, they are semi-abelian categories. 

For every left skew ring $(R,+,\circ)$, we can define the  {\em $0$-symmetric part} 
\[
R_0 := \{\,a\in R \mid 0\circ a = 0\,\}
\]
of $R$ and the {\em constant part}
\[
R_c := \{\,a\in R \mid 0\circ a = a\,\}.
\]

Every left skew ring $R$ decomposes as a semidirect sum $R = R_0 \rtimes R_c$ as an additive group, because $\lambda_0^R$ is an idempotent group endomorphism of $(R,+)$ (Proposition~\ref{2.3}(b)).

		Let $(R, +,\circ)$ be a left skew ring. An {\em ideal} of the left skew ring $(R, +,\circ)$ is a normal subgroup $I$ of the additive group $(R, +)$ such that $r\circ i-r\in I$ and $(i+r)\circ s-r\circ s\in I$ for every $r,s\in R$ and  $i\in I$. (Notice that, for a normal subgroup $I$ of the additive group $R$, $(r+i)\circ s-r\circ s\in I$ for every $i\in I$ if and only if $(i+r)\circ s-r\circ s\in I$ for every $i\in I$.)
		
		\begin{proposition} Let $(R, +,\circ)$ be a left skew ring. There is a one-to-one correspondence between the set of all ideals of $R$ and the set of all congruences of the left skew ring $R$.\end{proposition}
	
\section{Subtraction of operations, and left dirings}\label{4}

\subsection{The left near-ring structure on the set of operations on a group}\label{B(G)}

It is well known that for any (not-necessarily abelian) group $(G,+)$, the set $M(G)$ of all mappings from $G$ to $G$ is a right near-ring. Something similar occurs for binary operations on $G$. More precisely, if $S$ is any set, the set $B(S)$ of all binary operations on $S$ is a monoid with identity the projection map onto the first component, that is $\pi_1\colon S\times S\to S$, $\pi_1(a, b) =
a$ for all $a, b\in S$ \cite[Theorem~1.2]{LPOMR}. If $(G,+)$ is a group, then the set $B(G)$ of all binary operations on $G$ is a left near-ring with two-sided identity the first projection $\pi_1$ \cite[Theorem~2.4.5]{OM}. Here, if $\circ_1$ and $\circ_2$ are two operations on $G$, their sum is defined by $a(\circ_1+\circ_2)b=(a\circ_1 b)+(a\circ_2 b)$ and their product is $a(\circ_1\circ_2)b=(a\circ_1 b)\circ_2(a\circ_1 b)$ for every $a,b\in G$.
This is the addition on $B(V)$, for $V$ any module over a commutative ring $k$ in which $2$ is invertible, when one notices that every $k$-bilinear operation on $V$ is the sum of a $k$-bilinear commutative operation and an anticommutative one.

In the next subsection, we will write the identity $\pi_1$ of the left near-ring $B(G)$ as the difference of an associative operation and a left distributive operation. That is,  for a group  $(G,+)$, we will look for the pairs of operations $(\circ,\cdot)$ on $G$ with $\circ$ associative, $\cdot$ left distributive, and $\pi_1=\circ-\cdot$ in $B(G)$. 

\subsection{Left dirings} 

To this end, we now present a variation of the notion of left near-ring. Essentially, the multiplication of a left near-ring, and its properties, are now ``distributed" between two multiplications $\circ$ and $\cdot$.

\begin{definition} A {\em left diring} is an algebra $(G,+,-,0,\circ,\cdot)$, where  $+$, $\circ$ and $\cdot$ are three binary operations, $-$ is unary, and $0$ 
is nullary, satisfying the following conditions:

\begin{enumerate}
    \item[(a)] $(G,+,-,0)$ is a group, not-necessarily abelian; and 
    \item[(b)] the difference in $B(G)$ of the two operations $\circ$ and $\cdot$ is the operation $\pi_1$. 
\end{enumerate}
\end{definition} 

Condition (b) simply says that \begin{equation}a\circ b-a\cdot b=a\label{skew}\end{equation} for every $a,b\in G$.

\bigskip

\begin{proposition}\label{new} Let $(G,+,-,0,\circ,\cdot)$ be a left diring. Then:

\begin{enumerate}
    \item[(a)] $0\cdot b=0\circ b$ for every $b\in G$.
    \item[(b)] The operation $\cdot$ is left distributive, that is, $a\cdot(b+c)=a\cdot b+a\cdot c$ for all $a,b,c\in G$, if and only if the operation $\circ$ is {\em left skew distributive}, that is, $a\circ(b+c)=a\circ b-a+a\circ c$ for all $a,b,c\in G$.
\end{enumerate}
\end{proposition}

\begin{proof} (a) follows immediately from Identity (\ref{skew}).

(b) One has $$a\cdot(b+c)=a\cdot b+a\cdot c$$ if and only if $$-a+a\circ(b+c)=-a+a\circ b-a+a\circ c,$$ that is, if and only if $$a\circ(b+c)=a\circ b-a+a\circ c,$$ as desired.
\end{proof}

\begin{proposition}\label{2.3'}  Let $(G,+,-,0,\circ,\cdot)$ be a left diring and suppose that the equivalent conditions of {\rm Proposition \ref{new}(b)} hold, that is, assume that the operation $\cdot$ is left distributive. Then:

        \begin{enumerate}
            \item[(a)] For every $a\in G$, the mapping $\lambda^G_a\colon G\to G$, defined by
$\lambda^G_a(b)=a\cdot b$ for every $b\in G$, is a group endomorphism of the group $(G,+)$.
\item[(b)] $a\cdot 0=0$, $a\cdot(-b)=-(a\cdot b)$, and $a\circ 0=a$ for every $a,b\in G$.
\item[(c)] The operation $\circ$ is associative, that is, $(a\circ b)\circ c=a\circ(b\circ c)$ for all $a,b,c\in G$, if and only if the operation $\cdot$ is {\em left weakly associative}, that is, $(a+a\cdot b)\cdot c =a\cdot(b\cdot c)$ for all $a,b,c\in G$.
        \end{enumerate}
		\end{proposition}
		
		\begin{proof} (a) is just a restatement of the hypothesis that the operation $\cdot$ is left distributive. 
		
		(b) follows from (a) and Identity (\ref{skew}).
		
		(c) $(a\circ b)\circ c=a\circ(b\circ c)$  if and only if $(a+a\cdot b)+(a+a\cdot b)\cdot c=a+a\cdot(b+b\cdot c)$, that is, if and only if $a+a\cdot b+(a+a\cdot b)\cdot c=a+a\cdot b+a\cdot(b\cdot c)$. This is equivalent to $(a+a\cdot b)\cdot c=a\cdot(b\cdot c)$.
		
		\end{proof}		
		By Proposition \ref{2.3'}(b), the identity $0$ of the additive group $G$, is a right zero for the magma $(G,\cdot)$, and is a right identity for the magma $(G,\circ)$, provided that left distributivity of $\cdot$ holds.
	
		Notice that weak associativity $(a+ab)\cdot c =a\cdot(b\cdot c)$ can be equivalently written as $(a\circ b)c=a(bc)$ (and is this form the name weak associativity is justified), for all $a,b,c\in G$. Equivalently,  the two equivalent statements in Proposition \ref{2.3'}(c) say that the mapping $\lambda$ is a semigroup morphism of the semigroup $(G,\circ)$ into the semigroup $\End_\Gp(G,+)$ (with composition of endomorphisms). We include this fact as statement (a) in the following proposition.
		
		\begin{proposition}\label{2.3''}  Let $(G,+,-,0,\circ,\cdot)$ be a left diring and suppose that $\circ$ is associative and $\cdot$ is left distributive. Then:

        \begin{enumerate}
            \item[(a)] The mapping $\lambda^G\colon (G,\circ)\to \End_\Gp(G,+)$, defined by $\lambda^G\colon a\mapsto \lambda^G_a$, is a semigroup morphism. That is, $(a\circ b)\cdot c=a\cdot(b\cdot c)$ for every $a,b,c\in G$.
            \item[(b)] The group endomorphism $\lambda^G_0$ of the group $(G,+,-,0)$ is an idempotent group endomorphism.
        \end{enumerate}
		\end{proposition}
		
		\begin{proof} (a)  The position $a\mapsto \lambda^G_a$ defines a mapping of $G$ into $\End_\Gp(G,+)$ by Proposition~\ref{2.3'}(a). It is a semigroup morphism $(G,\circ)\to \End_\Gp(G,+)$ by Proposition~\ref{2.3'}(c).
		
				(b) We must show that $\lambda^G_0\circ\lambda^G_0=\lambda^G_0$, that is, that $0\cdot(0\cdot a)=0\cdot a$. But we have already remarked that $0\cdot a=0\circ a$ for all $a$, so that $0\cdot(0\cdot a)=0\cdot a$ is equivalent to $0\circ(0\circ a)=0\circ a$. Now $\circ$ is associative, hence $\lambda^G_0\circ\lambda^G_0=\lambda^G_0$ is equivalent to $(0\circ 0)\circ a=0\circ a$, which holds by Proposition \ref{2.3'}(b).
		\end{proof}
		
		Hence, suppose that $(G,+,-,0,\circ,\cdot)$ is a left diring with $\circ$ associative and $\cdot$ left distributive.  Then $\lambda^G_0$ is an idempotent group endomorphism of the additive group of $G$, hence it corresponds to a semidirect-sum decomposition of the group $(G,+)$. Now $$G_0:=\ker(\lambda^G_0)=\{\,a\in G\mid 0\cdot a=0\,\}=\{\,a\in G\mid 0\circ a=0\,\}$$ is a normal subgroup of $(G,+)$, $$G_c:=\lambda^G_0(G)=\{\,a\in G\mid 0\cdot a=a\,\}=\{\,a\in G\mid 0\circ a=a\,\}$$ is a subgroup of $(G,+)$, and $G$ is the semidirect sum $G=G_0\rtimes G_c$ as an additive group. We say that $G_0$ is the $0${\em -symmetric part }of the left diring $(G,+,-,0,\circ,\cdot)$, and $G_c$ is its {\em constant part}. Both $G_0$ and $G_c$ are subdirings of $(G,+,-,0,\circ,\cdot)$, in the sense that they are additive subgroup of $(G,+)$ and are closed for both operations $\circ$ and $\cdot$. Notice that $0$ is a two-sided identity for the left skew ring $(G_c,+,-,0,\circ)$, and $0$ is a two-sided zero for the magma $(G_0,\cdot)$.
		
		A {\em left  weak ring} is an algebra $(W, +,-,0,\cdot)$ in which
\begin{enumerate}
    \item[(a)] $(W, +,-,0)$ is a (not-necessarily abelian) group;
    \item[(b)] the binary operation $\cdot$ is {\em weakly associative}, that is  $(a+ab)c =a(bc)$ for all $a,b,c\in W$; and 
    \item[(c)] {\em left distributivity} holds, that is, $a(b+c)=ab+ac$ for every $a,b,c\in W$.
\end{enumerate}

Left weak rings are very similar to the {\em post-groups} defined in \cite[Definition~2.1]{Ann}, except for  the fact that in our left weak rings the mappings $\lambda_a$ defined by $\lambda_a(b)=a\cdot b$ for every $b$, need not to be group automorphisms of $(W,+)$, but only endomorphisms of $(W,+)$. 

\begin{theorem}\label{main} The category of left skew rings and the category of the left  weak rings 
are canonically isomorphic. The canonical isomorphism associates to every left skew ring $(R, +,\circ)$ the left weak ring $(R, +,\cdot)$, where $\cdot$ is the binary operation on $R$ defined setting $a\cdot b=-a+a\circ b$ for all $a,b\in R$, and associates to each left skew ring morphism $f\colon R\to S$ the same mapping~$f$. \end{theorem}
		
	\begin{proof} Let $(R, +,\circ)$ be a left skew ring. Let $\cdot$ be the binary operation on $R$ defined setting $a\cdot b=-a+a\circ b=\lambda_a(b)$ for all $a,b\in R$, so that $(R, +,\circ,\cdot)$ is a left diring in which $\circ$ is associative and $\cdot$ is left distributive. Then  $(R, +,\cdot)$ is a left weak ring by Propositions~\ref{new}(b) and \ref{2.3'}(c). Notice that a mapping $f\colon R\to S$, $S$ a left skew ring, respects $+$ and $\circ$ if and only if it respects $+$ and $\cdot$. Thus, we have a canonical functor $(R, +,\circ)\mapsto (R, +,\cdot)$ of the category of left skew rings into the category of left weak rings.
		
	Now let $(W, +,\cdot)$ be a left weak ring, and let $\circ$ be the binary operation on $W$ defined setting $a\circ b=a+a b$ for every $a,b\in W$. The operation $\circ$ is associative by Proposition~\ref{2.3'}(c) and left skew distributive by Propositions~\ref{new}(b). Therefore $(W,+,\circ)$ is a left skew ring, and we get a canonical functor $(W, +,\cdot)\mapsto (W,+,\circ)$ of  the category of left weak rings into the category of left skew rings, which is clearly an inverse of the functor described in the previous paragraph.
		\end{proof}
		
	\begin{corollary}\label{cuo} {\rm (cf.~\cite[Subsection~3.3]{Ann})} The category of left skew braces is isomorphic to  the category of the left weak rings $(R, +,\cdot)$ that have {\em local right identities}, that is, for every element $a$ in the left weak ring $R$, there exists an element $e$ (which may depend on $a$) such that $a\cdot e=a$.\end{corollary}
	
	\begin{proof} It is well known that a semigroup $(R,\circ)$ is a group if and only if 	$(R,\circ)$ has a right identity and every element of $(R,\circ)$ has a right inverse. By Proposition~\ref{2.3}(a), a left skew ring $(R,+,\circ)$ is a left skew brace if and only if every element of $(R,\circ)$ has a right inverse, that is, if and only if for every $a\in R$ there exists $b\in R$ such that $a\circ b=0$, i.e.~if and only if for every  $a\in R$ there exists $b\in R$ such that $a+ab=0$, or, equivalently, $-ab=a$. Replacing $b$ with its opposite $-b:=e$, we get that $(R,+,\circ)$ is a left skew brace if and only if, in the corresponding left weak ring $(R,+,\cdot)$, for every element $a\in R$ there exists an element $e$ such that $ae=a$.\end{proof}
			
			Summing up, we have seen that, in a left diring $(G,+,-,0,\circ,\cdot)$, essentially:

            \begin{enumerate}
                \item[(a)] The operation $\circ$ is  left skew distributive if and only if the operation $\cdot$ is left distributive (Proposition~\ref{new}(b)).
                \item[(b)] The operation $\circ$ is  associative if and only if the operation $\cdot$ is left weakly associative (Proposition~\ref{2.3'}(c)).
                \item[(c)]  $(G,+,-,0,\circ)$ is a left skew brace if and only every element has a local right identity with respect to $\cdot$ (Corollary~\ref{cuo}).
                \end{enumerate}

		This justifies the strange axiom of ``left skew distributivity," which at first sight strikes us so strongly when we encounter it the first time we see the definition of a left skew brace: it is nothing other than the usual distributivity of the other operation $\cdot$. Left skew braces turn out, moreover, to be a particularly rich non-additive algebraic structure from a categorical point of view---on par with groups and rings (they form a strongly protomodular category \cite[Theorem~4.3]{B})---and therefore much better behaved than many other algebraic structures one usually encounters.	
						
		\begin{example}\label{mmm} 	Let  $(G,+)$ be a group. Let us go back to the pairs of operations $(\circ,\cdot)$ on $G$ with $\circ$ associative, $\cdot$ left distributive, and $\pi_1=\circ-\cdot$ in $B(G)$ (Subsection~\ref{B(G)}). We already know that this implies that $\circ$ must be necessarily left skew distributive and $\cdot$ must be weakly associative (Propositions~\ref{new}(b) and \ref{2.3'}(c)).
Some of the most natural operations on a non-trivial group $(G,+)$, in my opinion, are the following:

\begin{enumerate}
    \item[(a.1)] The null operation $\circ_0$ defined by $a\circ_0 b=0_G$ for every $a,b\in G$. It is the zero in the left near-ring $B(G)$. The operation $\circ_0$ is associative, left weakly associative, commutative, (left and right) distributive, but not skew distributive. Thus $(G,+,\circ_0)$ is a commutative near-ring and a weak ring.
    \item[(b.1)] The identity $\pi_1$ of  the left near-ring $B(G)$, defined by $a\pi_1 b=a$ for every $a,b\in G$. The operation $\pi_1$ is associative, right distributive, and left skew distributive, but neither left weakly associative nor left distributive. Therefore, $(G,+,\pi_1)$ is a left skew ring.
    \item[(c.1)] The operation $\pi_2$ defined by $a\pi_2 b=b$ for every $a,b\in G$. The operation $\pi_2$ is associative, left weakly associative, and 
left distributive, but not left skew distributive. Hence $(G,+,\pi_2)$ is a left weak ring and a left near-ring.
\item[(d.1)] The operation $+=\pi_1+\pi_2$ on the group $G$. It is associative and left skew distributive, but not left weakly associative, and not
left distributive. Notice that $(G,+,+)$ is a left skew brace, hence a left skew ring.
\item[(e.1)] The operation $+^{\op}=\pi_2+\pi_1$ on the group $G$. It is associative and left skew distributive, but not left weakly associative, and not
left distributive. Again, $(G,+,+^{\op})$ is a left skew brace, hence a left skew ring.

\item[(f.1)]  $\cdot=-\pi_1+\pi_2+\pi_1$ (conjugation) for $(G,+)$ nonabelian, because if $G$ is abelian then $\cdot$ is $\pi_2$, which we have already studied in (3). It is not associative, but it is left distributive and left weakly associative.
\end{enumerate}

Correspondingly, we have three natural ways of writing $\pi_1$ as the difference of an associative operation $\circ$ and a left distributive operation $\cdot$ (of course, there are many more!):

\begin{enumerate}
    \item [(a.2)] $\pi_1=\pi_1-\circ_0$.  Here we find the the left skew ring is $(G,+,\pi_1)$ and the left weak ring is $(G,+,\circ_0)$ with $a\circ_0 b=0$ for every $a,b\in G$. The diring $(G,+,\pi_1,\circ_0)$ is $0$-symmetric.
    \item[(b.2)] $\pi_1=+-\pi_2$. In this case, we have that the left skew ring is $(G,+,+)$, which is a left skew brace, and the left weak ring is $(G,+,\pi_2)$ with $a\pi_2 b=b$ for every $a,b\in G$. The diring $(G,+,+,\pi_2)$ is equal to its constant part.
    \item[(c.2)] $\pi_1=+^{\op}-\pi_1-\pi_2+\pi_1=+^{\op}-(-\pi_1+\pi_2+\pi_1)$.  Here we find that the left skew ring is $(G,+,+^{\op})$, which is a left skew brace, and the corresponding left weak ring is $(G,+,\cdot)$ with $a\cdot b=-a+b+a$ for every $a,b\in G$ (the conjugate of $b$ via $a$). 
\end{enumerate}
    \end{example}

\begin{example} 
\begin{enumerate}
    \item[(a)] Let $(G,+,\circ,\cdot)$ be a $0$-symmetric diring with $\circ$ associative and $\cdot$ distributive. Then $0\cdot a=0$ for every $a\in G$. But $\cdot$ is weakly associative, so that $(a+a\cdot 0)c=a(0\cdot c)$ for all $a,c\in G$, from which $ac=0$. Therefore, we are in the case of Example~\ref{mmm}(a.2).

We have thus proved that Example~\ref{mmm}(a.2) is the unique example of $0$-symmetric diring with $\circ$ associative and $\cdot$ distributive. 

\item[(b)] If $(G,+,\circ,\cdot)$ is a diring with $\circ$ associative, $\cdot$ distributive and $G=G_c$, so that $0\circ a=a$ for every $a\in G$, then $(G,\circ, 0)$ is a monoid. Therefore, $(G,+,\circ)$ is a left skew ring with identity, the case considered by Rump in \cite[Corollary of Proposition~1]{RumpJAA2019}.
\end{enumerate}
\end{example}

\begin{example}\label{4.9} Let $(G,+)$ be any group and $e$ an idempotent endomorphism of $(G,+)$. Equivalently, suppose that $(G,+)$ has a semidirect-sum decomposition $G=K\rtimes H$ (the correspondence is such that $K=\ker(e)$ and $H=e(G)$). Define an operation $\cdot$ on $G$ setting $ab=e(b)$ for every $a,b\in G$. Then it is easily checked that $\cdot$ is left weakly associative and left distributive, so that $(G,+,\cdot)$ is a left weak ring. It is easily seen that $G_0=K$ and $G_c=H$.
\end{example}

\begin{remark} All the $\Omega$-groups $(G,+,-,0, p_a\mid a\in\Omega)$ considered in this paper are $\Omega$-groups in the sense of Higgins \cite{26}. Sometimes, in the literature, another, more restrictive notion of $\Omega$-group is studied. It further requires that the additional algebraic operations $p_a$ distribute over the group operations. These are called {\em distributive $\Omega$-groups} by Higgins \cite[p.~377]{26}. The $\Omega$-groups considered in this article (left near-rings, left skew rings, etc.) are not distributive $\Omega$-groups in general, because distributivity holds at most on the left and not necessarily on the right.

For example, let $(G,+,\circ)$ be a digroup. Then $(G,+,\circ)$ is an $\Omega$-group in the sense of Higgins. Moreover, if we define the multiplication $\cdot$ on $G$ setting $xy:=-x+x\circ y$, then $(G,+,\circ)$ is a left skew brace if and only if $\cdot$ is left distributive. Thus $(G,+,\circ)$ and $(G,+,\cdot)$ are $\Omega$-groups, but they are not distributive $\Omega$-groups. Of course, for $(G,+,\circ)$ a left skew brace, one could replace the operation $\cdot$ with all the unary operations $\lambda_a$, $a\in G$, where $\lambda_a(y)=a\cdot y$. Then $(G,+,-,0, \lambda_a\mid a\in G)$ turns out to be a distributive $\Omega$-group. It is also a {\em group with operators}, that is, a distributive $\Omega$-group $(G,+,-,0, p_a\mid a\in\Omega)$ in which all the operations $p_a$ are unary.
\end{remark}

Let me conclude this paper by indicating some possible directions for further research. One possible direction is to try to apply the ideas of this article to other algebraic structures. I have begun to investigate this in \cite{FF}. I am very grateful to the referees of this paper, who have suggested some other possible interesting lines of investigation. One is the following. Since one of the motivations for the study of skew braces is the Yang-Baxter equation, it would be
interesting to know whether the left-diring perspective suggests new ways
of constructing or classifying set-theoretic solutions, beyond the class already captured by skew braces. For instance, do certain classes of left
weak rings (not coming from groups via local right identities) give rise to
generalized or ``weak" solutions in an appropriate sense? Another idea is the following.
The examples built from idempotent endomorphisms (Example~\ref{4.9}) hint
at a connection with decompositions of groups and modules. Are there
potential applications to the representation theory of groups or to module
categories, for example, by interpreting $G_0$ and $G_c$ in terms of fixed points
and orbits under certain actions, or to the study of radical and semisimple
parts of near-rings and related objects?

 \subsubsection*{Acknowledgments}
		I am grateful to George Janelidze for useful suggestions on a preliminary version of this paper, and to Ma{\l}gorzata Hryniewicka, who drew my attention to the paper  \cite{Ann} during her talk at a conference in Lens.

% A table of contents will be automatically inserted in your article if it
% has 3 or more sections.  Please, do not try to manually change this
% behaviour.

% Also, please consider the following suggestions while preparing your 
% manuscript (as they will speed up the editorial process):
% * Avoid starting a new sentence with a mathematical formula;
% * Try to separate adjacent formulas with words;
% * Avoid inline formulas longer than half of a line. You can use math 
%   displays (\[...\]) instead;
% * Consider the use the enumerate and itemize environments for lists;
% * Consider the use of \dots, \ldots, \dotsc, \cdot, etc, instead of "..." 
%   or ".";
% * Instead of numbering or citing an article by hand (using parenthesis or 
%   brackets), consider the use of \cite, \ref and \eqref for citations and
%   cross-references;
% * Try to avoid inserting horizontal or vertical spacing, such as \hskip, 
%   \vskip and \bigskip;
% * Try to avoid inserting line or page brakes, such as \\, \newpage and
%   \clearpage.

% Acknowledgments should be added at the end of this section (right before
% the refences section) as a \subsection* (a subsection without a number):
% \subsection*{Acknowledgments} ...

%%% REFERENCES %%%

% Please, do not change the above line and do not insert your references
% into this file.  Instead, insert your references into the cimart.bib file.
% See cimart.bib for further instructions.

\EditInfo{October 1, 2025}{December 15, 2025}{David Towers and Ivan Kaygorodov}

\end{document}